\documentclass[11pt]{article}

\usepackage{graphicx}
\usepackage{amsthm,url}
\usepackage{hyperref}
\usepackage{theoremref}
\hypersetup{colorlinks}%
\usepackage[nottoc]{tocbibind}
\usepackage{tikz-cd}
\usepackage{diagbox}
\usepackage{amsmath,amsfonts,latexsym,amssymb,mathtools}
\usepackage{enumitem,tikz}
\usepackage{xcolor}
\usepackage{bbm}
\usepackage[ruled,vlined]{algorithm2e}

\newtheorem{thm}{Theorem}[section]
\newtheorem{cor}[thm]{Corollary}
\newtheorem{prop}[thm]{Proposition}
\newtheorem{lem}[thm]{Lemma}
\theoremstyle{definition}

\newtheorem{defi}[thm]{Definition}

\DeclareMathOperator{\Spec}{Spec}

\def\R{\mathbbm{R}}
\def\N{\mathbbm{N}}

\def\a{\alpha}
\def\y{\lambda}

\def\D{\Delta}
\def\ol{\overline}

\title{Hoffman colorings of graphs}

\author{Aida Abiad\thanks{\texttt{a.abiad.monge@tue.nl},  Department of Mathematics and Computer Science, Eindhoven University of Technology, The Netherlands}\thanks{Department of Mathematics and Data Science of Vrije Universiteit Brussels, Belgium}\qquad Wieb Bosma\thanks{\texttt{bosma@math.ru.nl}, Department of Mathematics, Radboud Universiteit Nijmegen, The Netherlands}  \qquad Thijs van Veluw\thanks{\texttt{thijs.vanveluw@ugent.be}, Department of Mathematics, Computer Science and Statistics, Ghent University, Belgium}}

\date{}

\begin{document}

\maketitle

%%%%%%%%%%%%%%%%%%%%%%%%%%%%%%%%%%%%%%%

\begin{abstract}
Hoffman's bound is a well-known spectral bound on the chromatic number of a graph, known to be tight for instance for bipartite graphs. While Hoffman colorings (colorings attaining the bound) were studied before for regular graphs, for general graphs not much is known. 
We investigate tightness of the Hoffman bound, with a particular focus on irregular graphs, obtaining several results on the graph structure of Hoffman colorings. In particular, we prove a Decomposition Theorem, which characterizes the structure of Hoffman colorings, and we use it to completely classify Hoffman colorability of cone graphs and line graphs. We also prove a partial converse, the Composition Theorem, leading to an algorithm for computing all connected Hoffman colorable graphs for some given number of vertices and colors. Since several graph coloring parameters are known to be sandwiched between the Hoffman bound and the chromatic number, as a byproduct of our results, we obtain the values of these chromatic parameters.\\


\end{abstract}

%%%%%%%%%%%%%%%%%%%%%%%%%%%%%%%%%%%%%%%

%%%%%%%%%%%%%%%%%%%%%%%%%%%%%%%%%%%%%%%%%%%%%%%%%%%%%%%%%%%%%%%%%%%%%%%%%%%%%%%%%%%%%%%%%%%%%%%%%%%%%%%%%%%%%

%%%%%%%%%%%%%%%%%%%%%%%%%%%%%%%%%%%%%%%%%%%%%%%%%%%%%%%%%%%%%%%%%%%%%%%%%%%%%%%%%%%%%%%%%%%%%%%%%%%%%%%%%%%%%
\section{Introduction}\label{sec:intro}
%%%%%%%%%%%%%%%%%%%%%%%%%%%%%%%%%%%%%%%%%%%%%%%%%%%%%%%%%%%%%%%%%%%%%%%%%%%%%%%%%%%%%%%%%%%%%%%%%%%%%%%%%%%%%

Consider a simple undirected graph $G=(V(G),E(G))$ of order $n$. The adjacency matrix $A$ of
a graph $G$ is the $n\times n$ matrix whose rows and columns are indexed by the vertices
of $G$, with entries satisfying $A_{u,v} = 1$ if $u\sim_G v$ (that is, if $\{u, v\}\in E(G)$) and $A_{u,v}=0$ otherwise.
This matrix is real and symmetric, and so its eigenvalues are real, and can be ordered
$\lambda_{\max}=\lambda_1 \geq \lambda_2 \geq  \cdots \geq \lambda_n=\lambda_{\min}$.

The \emph{independence number} of a graph $G$, denoted by $\alpha(G)$, is the size of a maximum coclique (also known as stable set or independent set) in a graph, and finding it is known to be in general an NP-hard problem. That is why spectral techniques have been used extensively to provide sharp bounds for the independence numbers of graphs, since the eigenvalues of a graph can be computed in polynomial time. There are two famous spectral bounds on the independence number of a graph: 
the inertia bound by Cvetkovi\'c \cite{inertia} and the ratio bound by Hoffman (unpublished; see \cite[Thm. 3.5.2]{spectra}).

\begin{thm}[Ratio bound, unpublished; see {\cite[Thm. 3.5.2]{spectra}}]\thlabel{thm:Hoffmanboundalpha}
    If $G$ is regular with $n$ vertices and largest and smallest adjacency eigenvalues $\lambda_{\max},\y_{\min}$ respectively, then
		$$\alpha(G)\leq n\frac{-\lambda_{\min}}{\lambda_{\max}-\lambda_{\min}},$$
		and if a coclique $C$ meets this bound, then every vertex not in $C$ is adjacent to precisely $-\lambda_{\min}$ vertices of $C$.
\end{thm}

Since all vertices with the same color form an independent set, upper bounds on the independence number $\alpha$ like the ones above directly yield lower bounds on the \emph{chromatic number} of a graph $\chi(G)$, which is the the smallest number of colors for $V(G)$ so that adjacent vertices are colored differently:
\begin{equation}\label{classicbound}
\chi\geq \lceil\frac{n}{\alpha}\rceil.
\end{equation}

Thus, the ratio bound on $\alpha$ from \thref{thm:Hoffmanboundalpha} gives a lower bound for the chromatic number of a regular graph. However, this bound was shown by Hoffman to hold also for general graphs.

\begin{thm}[Hoffman's bound, \cite{Hoffman}]\thlabel{thm:Hoffmanboundchi}
    Let $G$ be a non-empty graph. Then $$\displaystyle \chi(G) \ge 1-\frac{\y_{\max}(G)}{\y_{\min}(G)}.$$
\end{thm}

If the Hoffman bound (without rounding) of a graph is equal to its chromatic number, then we call the graph \emph{Hoffman colorable}, and every optimal coloring a \emph{Hoffman coloring}. Bipartite graphs and regular complete multipartite graphs are easily seen to be Hoffman colorable. For a non-trivial example, see Figure \ref{hoffmanirregular}. This graph has the surprising property that the Hoffman bound outperforms the bound from (\ref{classicbound}). Since it has $n=9$, $\a=5$, $\chi=3$, Hoffman's bound gives 3 and the classical bound (\ref{classicbound}) gives 2. And in fact, there are many more instances like this, which gives us the initial motivation to further investigate the tightness of Hoffman's bound.

\begin{figure}[ht]
    \begin{center}
        \begin{tikzpicture}[scale=0.4]
            \coordinate (1) at (-2,0);
\coordinate (2) at (0,-2);
\coordinate (3) at (0,2);
\coordinate (4) at (2,0);
\coordinate (5) at (0,0);
\coordinate (6) at (-1,-1);
\coordinate (7) at (-1,1);
\coordinate (8) at (1,-1);
\coordinate (9) at (1,1);
\draw[gray,thick] (1) -- (6);
\draw[gray,thick] (1) -- (7);
\draw[gray,thick] (2) -- (6);
\draw[gray,thick] (2) -- (8);
\draw[gray,thick] (3) -- (7);
\draw[gray,thick] (3) -- (9);
\draw[gray,thick] (4) -- (8);
\draw[gray,thick] (4) -- (9);
\draw[gray,thick] (5) -- (6);
\draw[gray,thick] (5) -- (7);
\draw[gray,thick] (5) -- (8);
\draw[gray,thick] (5) -- (9);
\draw[gray,thick] (6) -- (7);
\draw[gray,thick] (6) -- (8);
\draw[gray,thick] (7) -- (9);
\draw[gray,thick] (8) -- (9);
\filldraw[green] (1) circle (4pt);
\filldraw[green] (2) circle (4pt);
\filldraw[green] (3) circle (4pt);
\filldraw[green] (4) circle (4pt);
\filldraw[green] (5) circle (4pt);
\filldraw[red] (6) circle (4pt);
\filldraw[blue] (7) circle (4pt);
\filldraw[blue] (8) circle (4pt);
\filldraw[red] (9) circle (4pt);
        \end{tikzpicture}
    \end{center}
    \caption{Irregular graph meeting Hoffman's bound: $n=9$, $\a=5$, $\lambda_{\min}=-2, \lambda_{\max}=4$, $\chi=3$.}
    \label{hoffmanirregular}
\end{figure}
In a recent event to celebrate Hoffman's work, Haemers \cite{videoHaemers} suggested the problem of understanding which irregular graphs meet Hoffman's bound on the chromatic number. In \cite{spreads}, several strongly regular graphs are classified to be Hoffman colorable, but other than that not much on the structure of Hoffman colorings of general graphs was known. Investigating in which situations there is tightness indirectly also yields  conditions under which the Hoffman bound might be strengthened, as it was shown by Abiad \cite{Abiad}. Moreover, studying equality of the Hoffman bound also has another point of interest: various graph parameters are sandwiched in between the Hoffman bound and the chromatic number, for example the quantum chromatic number (see \cite{QuantumHom}) and the Lov\'{a}sz number of the complement (see \cite{Lovasz}), and for Hoffman colorable graphs, the values of these parameters are known immediately. See \cite{ParametersOverview} for an overview of such parameters.

Motivated by the above, in this paper we investigate tightness of the Hoffman bound, with a particular focus on irregular graphs. In particular, we show the Decomposition Theorem (\thref{thm:Decomp}), which provides structural requirements for Hoffman colorings with at least three colors. From the Decomposition Theorem we obtain multiple corollaries; we completely classify Hoffman colorability of cone graphs (\thref{thm:cone}) and line graphs (\thref{thm:linegraphs}). We also introduce a partial converse to the Decomposition Theorem, namely the Composition Theorem (\thref{thm:Comp}). The Composition Theorem shows under which conditions a Hoffman colorable graph is extendable with an additional color class to a bigger Hoffman colorable graph. As a consequence, we can prove Hoffman colorability of regular graphs with a structured coloring and a high enough valency (\thref{prop:colorcomplementregular}), providing exponentially many non-isomorphic regular Hoffman colorable graphs. The Decomposition and Composition Theorems enable us to devise an algorithm for computing all connected Hoffman colorable graphs on a given number of vertices and colors. Among the resulting connected Hoffman colorable graphs, we pay particular attention to the irregular graphs where the Hoffman bound outperforms the classical bound (\ref{classicbound}), like the Hoffman colorable graph from Figure \ref{hoffmanirregular}. In Tables \ref{tab:three}, \ref{tab:four}, \ref{tab:five}, and \ref{tab:six} we include the algorithmic results.

This article is organized as follows. In Section \ref{sec:preliminaries} we set out the preliminaries. Section \ref{sec:decomp} covers the Decomposition Theorem (\thref{thm:Decomp}) and its consequences. In Section \ref{sec:comp}, we state the Composition Theorem (\thref{thm:Comp}) and its application to regular graphs. Lastly, in Section \ref{sec:algorithm}, we introduce the algorithm and present the results of applying it.

%%%%%%%%%%%%%%%%%%%%%%%%%%%%%%%%%%%%%%%%%%%%%%%%%%%%%%%%%%%%%%%%%%%%%%%%%%%%%%%%%%%%%%%%%%%%%%%%%%%%%%%%%%%%%
\section{Preliminaries}\label{sec:preliminaries}
%%%%%%%%%%%%%%%%%%%%%%%%%%%%%%%%%%%%%%%%%%%%%%%%%%%%%%%%%%%%%%%%%%%%%%%%%%%%%%%%%%%%%%%%%%%%%%%%%%%%%%%%%%%%%
In this section, we establish basic notation, definitions, and background relating to positive eigenvectors and eigenvalue interlacing. We also cover some of the previous work that has been done on Hoffman colorings. For indicating Hoffman's bound we denote $$ \displaystyle h(G)\coloneqq 1-\frac{\y_{\max}(G)}{\y_{\min}(G)}.$$

%%%%%%%%%%%%%%%%%%%%%%%%%%%%%%%%%%%%%%%%%%%%%%%%%%%%%%%%%%%%%%%%%%%%%%%%%%%%%%%%%%%%%%%%%%%%%%%%%%%%%%%%%%%%%
\subsection{Perron-Frobenius and positive eigenvectors}
%%%%%%%%%%%%%%%%%%%%%%%%%%%%%%%%%%%%%%%%%%%%%%%%%%%%%%%%%%%%%%%%%%%%%%%%%%%%%%%%%%%%%%%%%%%%%%%%%%%%%%%%%%%%%

For disconnected graphs, Hoffman colorability can be deduced from the spectra and chromatic numbers of its connected components. In fact, a Hoffman colorable graph $G$ remains Hoffman colorable when it is disjointly extended by a graph $H$ as long as $\y_{\max}, -\y_{\min}$ and $\chi$ of $H$ are not larger than those of $G$. In particular, we can add as many isolated vertices as we want, artificially increasing the independence number while maintaining the Hoffman bound. In light of comparing the Hoffman bound to (\ref{classicbound}), it is therefore most interesting to look at connected graphs only.

By the Perron-Frobenius Theorem (\cite[Theorem 2.2.1]{spectra}), connected graphs have a (up to scaling) unique eigenvector with only positive entries (a \emph{positive eigenvector}), and it belongs to the largest eigenvalue. As we will see in Section \ref{sec:weight} and \ref{sec:proof}, our results are based on weight-interlacing, which interprets the positive eigenvector as weights belonging to the vertices.

A disconnected graph can also have a positive eigenvector; namely if the largest eigenvalues of all components are equal. So, our results apply more broadly than just to connected graphs. If a disconnected graph $G$ has a positive eigenvector, then it follows that the Hoffman bound of $G$ is equal to the minimum of the Hoffman bounds of its components. Similarly, the chromatic number of $G$ is equal to the maximum of the chromatic numbers of the components of $G$. So if $G$ is Hoffman colorable, then so is every component, and moreover every component shares the invariants $\y_{\max}, \y_{\min}$ and $\chi$ with $G$.

With this in mind, in the following we define a disconnected graph to be Hoffman colorable only if every connected component is Hoffman colorable and the values $\y_{\max}$, $\y_{\min}$, and $\chi$ are the same for every component. This way, we can always assume that a Hoffman colorable graph has a positive eigenvector.

%%%%%%%%%%%%%%%%%%%%%%%%%%%%%%%%%%%%%%%%%%%%%%%%%%%%%%%%%%%%%%%%%%%%%%%%%%%%%%%%%%%%%%%%%%%%%%%%%%%%%%%%%%%%%
\subsection{Interlacing}
%%%%%%%%%%%%%%%%%%%%%%%%%%%%%%%%%%%%%%%%%%%%%%%%%%%%%%%%%%%%%%%%%%%%%%%%%%%%%%%%%%%%%%%%%%%%%%%%%%%%%%%%%%%%%

Let $\y_1\ge \y_2\ge \dots \ge \y_n$ and $\mu_1\ge \mu_2 \ge \dots \ge \mu_m$ be two sequences of real numbers such that $m<n$. The latter sequence  \emph{interlaces} the former whenever
$$\y_i\ge \mu_i \ge \y_{n-m+i} \text{ for } i=1,\dots,m.$$
Interlacing is called \emph{tight} if there exists an integer $k \in \{0,\dots, m\}$ such that $\mu_i=\y_i$ for $i\le k$ and $\mu_i=\y_{n-m+i}$ for $i>k$.

In the context of interlacing of eigenvalues of matrices, we speak of \emph{eigenvalue interlacing}.
\begin{thm}[Interlacing Theorem, {\cite[Theorem 2.5.1: (i),(iv)]{spectra}}]\thlabel{thm:interlacing}
    Let $S$ be an $n\times m$-matrix such that $S^T S=I$. Let $A$ be a symmetric $n\times n$ matrix. Define $B=S^T AS$. Then the eigenvalues of $B$ interlace those of $A$. If this interlacing is tight, then $SB=AS$.
\end{thm}
A particular case of this is called Cauchy interlacing. In this case, we take $S$ such that the $m$ columns are independent unit vectors. Then $B$ is a principal submatrix of $A$.
\begin{cor}[Cauchy interlacing, {\cite[Corollary 2.5.2]{spectra}}]\thlabel{cor:indsubgr}
    If $B$ is a principal submatrix of $A$, then the eigenvalues of $B$ interlace those of $A$.
\end{cor}

One particularly interesting application of interlacing, is the application to partitions and (weight-)quotient matrices. Weight-quotient interlacing and weight-regularity are the main ingredients of \thref{prop:Abiad} (cf.~\cite[Proposition 5.3(i)]{Abiad}), the interlacing proof of the Hoffman bound (\cite{CPP}), and of \thref{lem:cwr}, leading to the Decomposition Theorem.

%%%%%%%%%%%%%%%%%%%%%%%%%%%%%%%%%%%%%%%%%%%%%%%%%%%%%%%%%%%%%%%%%%%%%%%%%%%%%%%%%%%%%%%%%%%%%%%%%%%%%%%%%%%%%
\subsection{Equitable partitions and quotients}
%%%%%%%%%%%%%%%%%%%%%%%%%%%%%%%%%%%%%%%%%%%%%%%%%%%%%%%%%%%%%%%%%%%%%%%%%%%%%%%%%%%%%%%%%%%%%%%%%%%%%%%%%%%%%

Let $G$ be a graph, and $\mathcal P$ a partition of $V(G)$ into $V_1,\dots,V_m$. The \emph{quotient matrix} given this partition is the matrix $B$ with entries
$$b_{ij}=\frac 1{|V_i|} \mathbbm 1^T A_{ij} \mathbbm 1,$$
with $\mathbbm 1$ the all-ones vector and $A_{ij}$ the principal submatrix of $A$ indexed by the vertices of $V_i$ and $V_j$. This way $b_{ij}$ is equal to the average row sum of $A_{ij}$. In other words, $b_{ij}$ is the average over the vertices in $V_i$ of the number of neighbors in $V_j$. A partition is \emph{equitable} or \emph{regular} if for each row of $A_{ij}$ the sum is equal to $b_{ij}$, or equivalently, if $A_{ij} \mathbbm 1$ is a constant vector. This is also equivalent to the number of neighbors in $V_j$ of a vertex in $V_i$ not depending on the specific vertex one chooses from $V_i$. In this case, the number $b_{ij}$ is called the \emph{intersection number}. The Interlacing Theorem implies the following.
\begin{cor}[Quotient matrix interlacing {\cite[Corollary 2.5.4]{spectra}}]\thlabel{cor:quotientmatrixinterlacing}
    Let $G$ be a graph, with $A$ its adjacency matrix. Let $\mathcal P$ a partition of $V(G)$, inducing the quotient matrix $B$. Then
    \begin{enumerate}[label=(\roman*)]
        \item the eigenvalues of $B$ interlace those of $A$;
        \item if the interlacing is tight, then the partition is equitable.
    \end{enumerate}
\end{cor}
For Hoffman colorings of regular graphs, the ratio bound immediately implies the following. Alternatively, it can be proved using \thref{cor:quotientmatrixinterlacing}.

\begin{prop}[{\cite[Proposition 2.3]{3chromDRG}}]\thlabel{prop:constantequitable}
    Let $V_1,\dots, V_\chi$ be a Hoffman coloring of a regular graph $G$, then
    \begin{enumerate}[label=(\roman*)]
        \item the partition $V_1,\dots, V_\chi$ is equitable;
        \item all intersection numbers $b_{ij}$ of this equitable partition with $i\ne j$ equal $-\y_{\min}(G)$;
        \item all color classes have equal size.
    \end{enumerate}
\end{prop}

%%%%%%%%%%%%%%%%%%%%%%%%%%%%%%%%%%%%%%%%%%%%%%%%%%%%%%%%%%%%%%%%%%%%%%%%%%%%%%%%%%%%%%%%%%%%%%%%%%%%%%%%%%%%%
\subsection{Weight-equitable partitions and -quotients}\label{sec:weight}
%%%%%%%%%%%%%%%%%%%%%%%%%%%%%%%%%%%%%%%%%%%%%%%%%%%%%%%%%%%%%%%%%%%%%%%%%%%%%%%%%%%%%%%%%%%%%%%%%%%%%%%%%%%%%

In \cite{Fiol}, Fiol uses weights, coming from a positive eigenvector of the graph, to ``regularize'' irregular graphs. One interesting consequence is that we can consider weight-quotient matrices. Again, let $\mathcal P$ be a partition. Suppose $x$ is a positive eigenvector of $G$. Write $y_i$ for the restriction of $x$ onto color class $V_i$. So $y_i$ is a vector indexed by the vertices in $V_i$, and $y_i(u)=x(u)$ for $u\in V_i$. Then the \emph{weight-quotient matrix} given the partition is the matrix $B^*$ with entries
$$b^*_{ij}=\frac 1{||y_i||^2} y_i ^T A_{ij} y_j,$$
which is a weighted average of the weighted row sums. A partition is \emph{weight-equitable} or \emph{weight-regular} if the vector $A_{ij} y_j$ is a scalar multiple of $y_i$ for every $i$ and $j$. In this case, this scalar $b^*_{ij}$ is called the \emph{weight-intersection number}. For every pair of distinct colors $i,j$ we must have in this case
$$x(u) b^*_{ij}=\sum_{v\in N_j(u)} x(v),$$
where $u$ is a vertex of $V_i$, and $N_j(u)$ is the set of neighbors of $u$ in $V_j$. If $x$ is a constant vector, then all of the above is equivalent to an equitable partition.

We can again use interlacing for weight-equitability, to get the following statement, which is analogous to \thref{cor:quotientmatrixinterlacing}.
\begin{cor}[{\cite[Lemma 2.3]{Fiol}}]\thlabel{lem:weightpartition}
    Let $G$ be a graph with adjacency matrix $A$. Let $\mathcal P$ be a partition of $V(G)$, inducing the weight-quotient matrix $B$ given a positive eigenvector. Then
    \begin{enumerate}[label=(\roman*)]
        \item the eigenvalues of $B$ interlace those of $A$;
        \item if the interlacing is tight, then the partition is weight-equitable.
    \end{enumerate}
\end{cor}
Abiad \cite{Abiad} uses Corollary \ref{lem:weightpartition} to partially extend \thref{prop:constantequitable} to irregular graphs, and used it to prove the following necessary condition on Hoffman colorings.
\begin{prop}[{\cite[Proposition 5.3(i)]{Abiad}}]\thlabel{prop:Abiad}
    The partition defined by a Hoffman coloring is weight-equitable.
\end{prop}
The power of \thref{prop:Abiad} is that it generalizes the (weight-)equitability of the color partition to irregular graphs, extending (i) of \thref{prop:constantequitable} to irregular graphs. In \thref{lem:cwr}, we extend the rest of \thref{prop:constantequitable} to irregular graphs as well.

%%%%%%%%%%%%%%%%%%%%%%%%%%%%%%%%%%%%%%%%%%%%%%%%%%%%%%%%%%%%%%%%%%%%%%%%%%%%%%%%%%%%%%%%%%%%%%%%%%%%%%%%%%%%%
\section{The Decomposition Theorem and its consequences}\label{sec:decomp}
%%%%%%%%%%%%%%%%%%%%%%%%%%%%%%%%%%%%%%%%%%%%%%%%%%%%%%%%%%%%%%%%%%%%%%%%%%%%%%%%%%%%%%%%%%%%%%%%%%%%%%%%%%%%%

The Decomposition Theorem gives a necessary condition on the structure of Hoffman colorings. Where 2-colorings of graphs are automatically Hoffman colorings without any conditions, \thref{lem:cwr} and the Decomposition Theorem provide a condition that Hoffman colorings with more colors must satisfy. The Decomposition Theorem has various applications. In particular, we can classify Hoffman colorability of line graphs and cone graphs.

%%%%%%%%%%%%%%%%%%%%%%%%%%%%%%%%%%%%%%%%%%%%%%%%%%%%%%%%%%%%%%%%%%%%%%%%%%%%%%%%%%%%%%%%%%%%%%%%%%%%%%%%%%%%%
\subsection{The Decomposition Theorem}\label{sec:proof}
%%%%%%%%%%%%%%%%%%%%%%%%%%%%%%%%%%%%%%%%%%%%%%%%%%%%%%%%%%%%%%%%%%%%%%%%%%%%%%%%%%%%%%%%%%%%%%%%%%%%%%%%%%%%%

As mentioned before, we extend \thref{prop:Abiad} (\cite[Proposition 5.3(i)]{Abiad}) mirroring \thref{prop:constantequitable}.
\begin{thm}\thlabel{lem:cwr}
Let $G$ be Hoffman colorable with coloring $V(G)=\bigsqcup_{i=1}^\chi V_i$. Then
\begin{enumerate}[label=(\roman*)]
    \item the partition $V_1,\dots,V_{\chi}$ is weight-regular;
    \item all irreflexive intersection numbers $b^*_{ij}$ of this weight-equitable partition equal $-\y_{\min}(G)$;
    \item the restrictions of the positive eigenvector of $G$ onto the color classes have norm independent of the color class.
\end{enumerate}
\end{thm}

Part (i) is precisely \cite[Proposition 5.3(i)]{Abiad} (\thref{prop:Abiad}), and \thref{prop:constantequitable} is equivalent to \thref{lem:cwr} restricted to regular graphs. Note that Part (iii) is satisfied by any bipartite graph, but not necessarily by non-bipartite graphs.

In order to prove \thref{lem:cwr}, we use the following lemma.

\begin{lem}\thlabel{lem:J-Imultiple}
    Let $M$ be a symmetric real matrix that has zeroes on the diagonal and has two eigenvalues of which one is simple and positive. If $M$ only has non-negative entries, then $M$ is a scalar multiple of $J-I$, where $J$ is the all-ones matrix. Moreover, this scalar is equal to the absolute value of the least eigenvalue of $M$.
\end{lem}
\begin{proof}
Let $n$ be the size of $M$. Since the trace of $M$ is 0, we can find a positive real number $\nu$ such that the spectrum of $M$ is given by
$$\Spec(M)=\{(n-1)\nu,(-\nu)^{n-1}\}.$$
Let $x$ be an eigenvector for $(n-1)\nu$ of norm 1 such that $x_1\ge 0$. Then $x$ generates the eigenspace for $(n-1)\nu$, and the space of vectors orthogonal to $x$ forms the eigenspace for $-\nu$. The projection onto the eigenspace for $(n-1)\nu$ is now given by $x x^T$, and hence the projection onto the eigenspace for $-\nu$ is $I-xx^T$. By spectral decomposition, we obtain
$$M=(n-1)\nu\cdot  xx^T -\nu(I-x x^T)=\nu(n \cdot x x^T - I).$$
At the diagonal we thus have $M_{i,i}=\nu(nx_i^2-1)=0$. We conclude that $x_i=\pm 1/\sqrt n$ for all $i$. For values off the diagonal, we get $M_{i,j}=n\nu x_ix_j$. Since $n$ and $\nu$ are positive and $M$ only has non-negative entries, $x_i$ and $x_j$ must be of the same sign. We conclude that $x$ is a constant vector, so for all $i$ we have $x_i=x_1>0$. Now $xx^T= x_1^2 J=\frac 1n J$, and so $M=\nu (J-I)$.
\end{proof}

We are now ready to prove \thref{lem:cwr}. For the sake of completeness, we include the proof of \cite[Proposition 5.3(i)]{Abiad} as well. Moreover, along the way we see a proof of the Hoffman bound (\thref{thm:Hoffmanboundchi}) using interlacing, see \cite{CPP} or \cite{ProefschriftHaemers}.

\begin{proof}[Proof of \thref{lem:cwr}]
Let $x$ be the positive eigenvector of $G$ and write $y(i)=\sqrt{\sum_{v\in V_i} x(v)^2}$, the norm of the restriction of $x$ onto color class $V_i$. Let $S$ be the weight-characteristic matrix, which is a $ |V|\times \chi$-matrix given by
\[
S_{v,i}=\begin{cases}
    \frac{x(v)}{y(i)} &$if $v\in V_i,\\
    0 &$otherwise$.
\end{cases}
\]
Then note that $S^TS=I_\chi$, so now we can apply \thref{thm:interlacing} to get that the eigenvalues of $B=S^TAS$ interlace those of $A$. The vector $y$ with $y(i)$ defined as before is an eigenvector of $B$ with eigenvalue $\y_{\max}(A)$, because $(Sy)(v)=x(v)$ for all $v$. Also, $B$ has zeroes on the diagonal, since the color classes are independent sets. So the trace of $B$ is 0, hence the sum of the eigenvalues of $B$ is 0. Since the eigenvalues of $B$ interlace those of $A$, all the eigenvalues of $B$ are bounded below by $\y_{\min}(A)$, and hence
$$0=\text{tr}(B)\ge \y_{\max}(A) +(\chi-1) \y_{\min}(A).$$
Reordering gives the Hoffman bound, and this is exactly the interlacing proof. Since the graph $G$ is Hoffman colorable, this inequality is an equality, so that all eigenvalues of $B$, except the largest one, are equal to $\y_{\min}(A)$. We can then conclude two things:
\begin{enumerate}[label=(\arabic*)]
    \item The interlacing of $B$ and $A$ is tight,
    \item $B$ satisfies the requirements of \thref{lem:J-Imultiple}.
\end{enumerate}
The first point gives weight-equitability by \thref{lem:weightpartition}, and this is Abiad's proof of \cite[Proposition 5.3(i)]{Abiad}. The matrix $B$ now contains all the weight-intersection numbers.

However, by the second point, $B$ is a scalar multiple of $J-I$, say $B=\nu \cdot (J-I)$, where $\nu=-\y_{\min}(B)$, which by tight interlacing is equal to $-\y_{\min}(A)$, proving the second part of this result. The vector $y$ is an eigenvector to the largest eigenvalue of $B$, which is constant, so $y(i)$ is independent of $i$. This is Part (iii) of this result.
\end{proof}

Even though \thref{lem:cwr} is interesting in itself, we can go further and introduce the Decomposition Theorem, providing a new necessary condition on Hoffman colorings.

\begin{thm}[Decomposition Theorem]\thlabel{thm:Decomp}
    Let $G$ be Hoffman colorable with coloring $V(G)=\bigsqcup_{i=1}^\chi V_i$. Let $C$ be a subset of the colors $\{1,\dots,\chi\}$ with $|C|\ge 2$. Let $H$ be the induced subgraph of $G$ on the vertices $\bigcup_{i\in C} V_i$. Then the following hold.
    \begin{enumerate}[label=(\roman*)]
        \item $H$ is Hoffman colorable, with coloring $V(H)=\bigsqcup_{i \in C} V_i$;
        \item $\displaystyle\y_{\max}(H)=\frac{|C|-1}{\chi-1}\y_{\max}(G)$;
        \item The restriction of the positive eigenvector of $G$ to $H$ is a positive eigenvector of $H$, and consequently its eigenvalue is $\y_{\max}(H)$;
        \item $\y_{\min}(H)=\y_{\min}(G)$.
    \end{enumerate}
\end{thm}

\begin{proof}
Write $x$ for the positive eigenvector of $G$, and write $\nu=-\y_{\min}(G)$. Note that by \thref{lem:cwr} for all $i,j\in \{1,\dots,\chi\}$ with $i\ne j$ and all $u\in V_i$ we have
\begin{align*}
\nu x(u)&=\sum_{v\in N_j(u)} x(v),
\intertext{where $N_j(u)$ is the set of neighbors of $u$ of color $j$. If we fix $i$ in $C$ and $u\in V_i$ and let $j$ run over all other indices in $C$, then we get}
\nu(|C|-1)x(u)&=\sum_{j\in C\setminus \{i\}}\sum_{v\in N_j(u)} x(v)=\sum_{v\in N_H(u)} x(v),
\intertext{so $x | _{V(H)}$ is an eigenvector of $H$ with eigenvalue $\nu(|C|-1)$. This is a positive eigenvector, so its eigenvalue is the largest of $H$. Since $H$ is an induced subgraph of $G$, by Cauchy interlacing (see \thref{cor:indsubgr}) we have $\y_{\min}(H)\ge -\nu$. Looking at the Hoffman bound of $H$, we get}
h(H)&=1-\frac{\y_{\max}(H)}{\y_{\min}(H)}\ge  1+\frac{\nu(|C|-1)}\nu=|C|.
\intertext{The coloring $V(H)=\bigsqcup_{i\in C} V_i$ uses $|C|$ colors, so this must be a Hoffman coloring. Now $\y_{\min}(H)=-\nu=\y_{\min}(G)$ and by Hoffman colorability of $G$ we get}
\y_{\max}(H)&=\nu(|C|-1)=\frac{|C|-1}{\chi-1}\y_{\max}(G),
\end{align*}
which concludes the proof.
\end{proof}

%%%%%%%%%%%%%%%%%%%%%%%%%%%%%%%%%%%%%%%%%%%%%%%%%%%%%%%%%%%%%%%%%%%%%%%%%%%%%%%%%%%%%%%%%%%%%%%%%%%%%%%%%%%%%
\subsection{Some preliminary corollaries of the Decomposition Theorem}
%%%%%%%%%%%%%%%%%%%%%%%%%%%%%%%%%%%%%%%%%%%%%%%%%%%%%%%%%%%%%%%%%%%%%%%%%%%%%%%%%%%%%%%%%%%%%%%%%%%%%%%%%%%%%
Before we look at the consequences of the Decomposition Theorem as set out before, we state some very elementary corollaries that we use throughout.

In the following, a \emph{bipartite part} of a graph $G$ is an induced subgraph on two color classes of a coloring of $G$. If not specified, this coloring is meant to be a Hoffman coloring.

\begin{cor}\thlabel{cor:noiso}
    If $G$ is Hoffman colorable and $H$ is a bipartite part, then $H$ does not have isolated vertices.
\end{cor}
\begin{proof}
    Since $G$ is non-empty, $\y_{\max}(G)>0$. By the Decomposition Theorem, $H$ now has a positive eigenvector $x$ for a positive eigenvalue $\nu$. If $v\in V(H)$ is isolated, then $(A(H) \cdot x)(v)=0$, while it should be $\nu x(v)>0$.
\end{proof}
\begin{cor}\thlabel{cor:minimumdegree}
    Let $G$ be Hoffman colorable. Then for every optimal coloring of $G$ every vertex must have at least one neighbor of every other color.
\end{cor}
This means that connected graphs with a small minimal degree (compared to the chromatic number) cannot have a Hoffman coloring. In particular, connected non-bipartite Hoffman colorable graphs cannot have leaves. We also have the following, which is interesting in the context of cone graphs.
\begin{cor}\thlabel{cor:colorclasssize1}
  Any vertex that constitutes a color class on its own in a Hoffman coloring is a universal vertex, meaning that it is adjacent to every other vertex in the graph.
\end{cor}

%%%%%%%%%%%%%%%%%%%%%%%%%%%%%%%%%%%%%%%%%%%%%%%%%%%%%%%%%%%%%%%%%%%%%%%%%%%%%%%%%%%%%%%%%%%%%%%%%%%%%%%%%%%%%
\subsection{Hoffman colorability of cone graphs}
%%%%%%%%%%%%%%%%%%%%%%%%%%%%%%%%%%%%%%%%%%%%%%%%%%%%%%%%%%%%%%%%%%%%%%%%%%%%%%%%%%%%%%%%%%%%%%%%%%%%%%%%%%%%%

We are now ready to completely classify Hoffman colorability of cone graphs. This is important, because of three reasons: (1) several of the smallest Hoffman colorable graphs are cone graphs, (2) this class of graphs is suited very well to showing the power of the Decomposition Theorem, and (3) because of \thref{cor:colorclasssize1} this is not only a classification of Hoffman colorable cone graphs, but also of Hoffman colorable graphs with a color class of size 1.

It turns out that a similar reasoning works for a small generalization (what we decided to call \emph{$k$-cone graphs}), so we include it here as well.

\begin{defi}\thlabel{defi:cone}
The \emph{cone graph} over a graph $G$ is the graph obtained by adjoining a vertex to $G$ and connecting it to every vertex in $G$. For $k\in \N$, we define the \emph{$k$-cone graph} over $G$ as the graph obtained by adding a coclique of size $k$ to $G$ and connecting every vertex of the $k$-coclique to every vertex of $G$.
\end{defi}

\begin{thm}\thlabel{thm:cone}
Let $G$ be a non-empty graph. Then the $k$-cone graph over $G$ is Hoffman colorable if and only if $G$ is a regular Hoffman colorable graph with color classes of size $\y_{\min}(G)^2/k$.
    \end{thm}
    \begin{proof}
        Write $CG$ for the $k$-cone graph over $G$. Note that $\chi(CG)=1+\chi(G)\ge 3$, and that $CG$ is connected.
        
        Suppose that $CG$ is Hoffman colorable. Consider the optimal coloring $V(CG)=\bigsqcup_{i=0}^{\chi(G)} V_i$ where $V_0$ is the $k$-coclique added by the cone construction, and write $\nu=-\y_{\min}(CG)$. Let $x$ be the Perron eigenvector of $CG$. Consider for $1\le i \le \chi(G)$ the bipartite part $H_{0,i}$ on color classes with index 0 and $i$ of $CG$. By construction this is a complete bipartite graph, so $\nu=\sqrt {k |V_i|}$ by \cite[Section 1.4.2]{spectra} and the Decomposition Theorem, and $x$ must be constant on $V_i$. Note that all color classes must now be of the same size, namely $\nu^2/k$, and furthermore $x$ must be constant on $V(G)$. By the Decomposition Theorem, $G$ now has a constant eigenvector, implying that $G$ is regular.

        Conversely, suppose $G$ is regular and Hoffman colorable with all color classes of size $\nu^2/k$ (where $\nu=-\y_{\min}(G)$). Every eigenvector of $G$ orthogonal to the constant vector can be extended to an eigenvector of $CG$ with the same eigenvalue by setting 0 on the vertices in $V_0$. If $(r_i)_{i=1}^k$ is a sequence of real numbers adding to zero, then assigning $r_i$ to the $i$'th vertex of $V_0$ and 0 everywhere else results in an eigenvector of eigenvalue 0. Lastly, assigning $\nu$ to every vertex of $G$ and $\nu^2/k$ to $V_0$ gives an eigenvector of eigenvalue $\chi(G)\nu$ and the vector assigning $\nu$ to every vertex of $G$ and $-\chi(G)\nu^2/k$ to $V_0$ gives an eigenvector of eigenvalue $-\nu$. Therefore we have
        $$\Spec(CG)=\{\chi(G)\nu, 0^{k-1},-\nu\} \cup \Spec(G) \setminus \{(\chi(G)-1)\nu\}.$$
        The least eigenvalue of $G$ is $-\nu$, so we obtain that $CG$ is Hoffman colorable.
    \end{proof}
Alternatively, the ``conversely''-part can be proved using the Composition Theorem (\thref{thm:Comp}).

If we choose $\nu=1$, then this construction leads to all the complete graphs. For some more examples, we can take $G$ in \thref{thm:cone} to be bipartite, so that Hoffman colorability of $G$ is unconditional. If we write $\nu$ for the valency of $G$, then the condition of \thref{thm:cone} is that $G$ has order $2\nu^2/k$. For example, if $G$ is a regular bipartite graph on eight vertices with valency 2, then the cone over $G$ is Hoffman colorable (see Figure \ref{fig:conegraphs}). This can be done for every choice of $\nu$ and $k$ such that $\nu^2/k$ is an integer, providing infinitely many non-trivial Hoffman colorable $k$-cone graphs.

\begin{figure}[ht]
    \begin{center}
\begin{tikzpicture}[scale=0.8]
        \draw[gray, thick] (1,0) -- (0.7,0.7) -- (0,1) -- (-0.7,0.7) -- (-1,0) -- (-0.7,-0.7) -- (0,-1) -- (0.7,-0.7) -- (1,0) -- (-1,0);
        \draw[gray, thick] (0,1) -- (0,-1);
        \draw[gray, thick] (0.7,0.7) -- (-0.7,-0.7);
        \draw[gray, thick] (-0.7,0.7) -- (0.7,-0.7);
        \filldraw[green] (0,0) circle (2pt);
        \filldraw[red] (-1,0) circle (2pt);
        \filldraw[red] (1,0) circle (2pt);
        \filldraw[red] (0,1) circle (2pt);
        \filldraw[red] (0,-1) circle (2pt);
        \filldraw[blue] (0.7,0.7) circle (2pt);
        \filldraw[blue] (-0.7,0.7) circle (2pt);
        \filldraw[blue] (-0.7,-0.7) circle (2pt);
        \filldraw[blue] (0.7,-0.7) circle (2pt);
        
        \draw[gray, thick] (4,-0.5) -- (4,0.5) -- (3,0.5) -- (3,-0.5) -- (4,-0.5) -- (6,0.5) -- (7,0.5) -- (7,-0.5) -- (6,-0.5) -- (6,0.5);
        \draw[gray, thick] (4,0.5) -- (6,-0.5);
        \draw[gray, thick] (3,0.5) .. controls (4,1.5) and (4.5,1) .. (5,0) .. controls (5.5,1) and (6,1.5) .. (7,0.5);
        \draw[gray, thick] (3,-0.5) .. controls (4,-1.5) and (4.5,-1) .. (5,0) .. controls (5.5,-1) and (6,-1.5) .. (7,-0.5);
        \filldraw[green] (5,0) circle (2pt);
        \filldraw[red] (3,0.5) circle (2pt);
        \filldraw[blue] (4,0.5) circle (2pt);
        \filldraw[red] (4,-0.5) circle (2pt);
        \filldraw[blue] (3,-0.5) circle (2pt);
        \filldraw[red] (7,0.5) circle (2pt);
        \filldraw[blue] (6,0.5) circle (2pt);
        \filldraw[red] (6,-0.5) circle (2pt);
        \filldraw[blue] (7,-0.5) circle (2pt);
    \end{tikzpicture}
    \end{center}
    \caption{The cone graphs over the 8-cycle and the disjoint union of two 4-cycles.}
    \label{fig:conegraphs}
\end{figure}

%%%%%%%%%%%%%%%%%%%%%%%%%%%%%%%%%%%%%%%%%%%%%%%%%%%%%%%%%%%%%%%%%%%%%%%%%%%%%%%%%%%%%%%%%%%%%%%%%%%%%%%%%%%%%
\subsection{Hoffman colorability of line graphs}
%%%%%%%%%%%%%%%%%%%%%%%%%%%%%%%%%%%%%%%%%%%%%%%%%%%%%%%%%%%%%%%%%%%%%%%%%%%%%%%%%%%%%%%%%%%%%%%%%%%%%%%%%%%%%

Another subclass of graphs where the Decomposition Theorem allows us to completely classify Hoffman colorability, is the class of connected line graphs.

If $G$ is a graph, then we define the \emph{line graph} $L(G)$ such that $V(L(G))=E(G)$, and two edges are adjacent in the line graph if they share a vertex. The line graph $L(G)$ is connected if and only if $G$ is connected (apart from isolated vertices). Write $N$ for the incidence matrix of $G$, so that $N_{v,e}=1$ if vertex $v$ lies on edge $e$ and 0 otherwise. We have the following proposition.
\begin{prop}[{\cite[Proposition 1.4.1]{spectra}}]\thlabel{prop:linegraphleastev}
    Suppose $G$ has $m$ edges, and let $\rho_1 \ge \dots \ge \rho_r$ be the positive eigenvalues of $N N^T$. Then the eigenvalues of $L(G)$ are $\y_i=\rho_i-2$ for $i=1,\dots, r$ and $\y_i=-2$ for $i=r+1,\dots, m$.
\end{prop}
In particular, every eigenvalue of a line graph is at least $-2$.

The chromatic number of the line graph of $G$ is called the \emph{edge chromatic number} or the \emph{chromatic index} of $G$. By Vizing's famous result in \cite{Vizing}, the chromatic index of a graph is either $\D$ or $\D+1$, where $\D$ is the maximum degree. A graph where the chromatic index is equal to $\D$ is called \emph{class 1} and a graph with chromatic index equal to $\D+1$ is \emph{class 2}. Examples of class 1 graphs are $K_{2m}$ and $K_{m,m}$. Examples of class 2 graphs are regular graphs of odd order, including complete graphs $K_{2m+1}$.

A \emph{1-factor} or \emph{perfect matching} of a graph is a set of edges such that every vertex is on exactly one of those edges. A \emph{1-factorization} of a graph is a partition of the edges into 1-factors. A graph is \emph{1-factorable} if it admits a 1-factorization. Those graphs are necessarily regular, and have an even number of vertices. A 1-factorization is necessarily a coloring of the edges with just $\D$ colors. Therefore 1-factorable graphs are of class 1. Conversely, every regular graph of class 1 must be 1-factorable.

In case the least eigenvalue is equal to $-2$, we can determine Hoffman colorability. In fact, we see that we only get 1-factorable graphs in this case.
\begin{prop}[Hoffman colorability of line graphs with least eigenvalue $-2$]\thlabel{prop:1factor}
    Let $G$ be a connected graph with at least two edges, such that $\y_{\min}(L(G))=-2$. Then $L(G)$ is Hoffman colorable if and only if $G$ is 1-factorable.
\end{prop}
\begin{proof}
    Write $Q$ for the signless Laplace matrix of $G$, so $Q=NN^T$. We have $Q=D+A$, where $D$ is the diagonal matrix recording the degrees of the vertices of $G$, and $A$ is the adjacency matrix of $G$. We get $\y_{\max}(D)=\D$, the maximum degree, and by \cite[Proposition 3.1.2]{spectra} we have $\y_{\max}(A)\le \D$ with equality if and only if $G$ is regular. By linearity of the Rayleigh quotient (\cite[Section 2.4]{spectra}), we have $\y_{\max}(Q)\le \y_{\max}(D)+\y_{\max}(A)$. By \cite[Proposition 1.4.1]{spectra}
    $$\y_{\max}(L(G))=\y_{\max}(Q)-2\le \y_{\max}(D)+\y_{\max}(A)-2 \le 2\D-2,$$
    and so $h(L(G))\le \D$. By Vizing's Theorem, $\chi(L(G))\in \{\D,\D+1\}$. So we have
    $$h(L(G))\le \D \le \chi(L(G)).$$
    Now it is evident that $L(G)$ is Hoffman colorable if and only if both inequalities are equalities. The first equality is equivalent to $G$ being regular. The second equality is equivalent to $G$ being of class 1. Together, regularity and class 1 are equivalent to 1-factorability, concluding the proof.
\end{proof}
Now we have classified Hoffman colorablility of line graphs with least eigenvalue equal to $-2$, we should investigate which line graphs have this property. If a graph has more edges than vertices, then $r<m$ (with $r$ and $m$ taken from \thref{prop:linegraphleastev}) and so we must have an eigenvalue $-2$. If $G$ is connected and has an equal number of vertices and edges, then $G$ has a unique cycle. If this cycle is of even length, then consider the vector $x:E(G)\to \R$ assigning 0 to edges not in the cycle, and alternatingly $1$ and $-1$ on the edges in the cycle. This vector is an eigenvector of $L(G)$ and the eigenvalue is $-2$, so in this case \thref{prop:1factor} applies as well.

For the remaining cases (assuming that $G$ is connected), where $G$ is a tree, and where $G$ has a unique cycle of odd length, we need the Decomposition Theorem. Before we can solve these last cases, we state a lemma. It applies to line graphs, but we state it in the most general way possible.
\begin{lem}\thlabel{cor:LineGraphComponent}
    Let $G$ be Hoffman colorable. Suppose that $G$ has an optimal coloring $V(G)=\bigsqcup_{i=1}^\chi V_i$ such that every vertex is adjacent to at most two vertices of every color. Let $C=\{j,j'\}$ be a pair of colors: $C\subseteq \{1,\dots, \chi\}$ with $|C|=2$. Let $H$ be the induced subgraph of $G$ on $V_i \cup V_j$. Let $K$ be a connected component of $H$. Then exactly one of the following holds.
    \begin{enumerate}[label=(\roman*)]
        \item $\y_{\min}(G)=-2$ and $K$ is an even cycle;
        \item There exists a positive integer $m$ such that $\y_{\min}(G)=-2\cos \Big ( \frac\pi{m+1}\Big )$ and $K$ is a path on $m$ vertices.
    \end{enumerate}
\end{lem}
\begin{proof}
    By the assumption on the coloring, $K$ has maximum degree 2. Then $K$ has to be a cycle or a path. Since $K$ is bipartite, if $K$ is a cycle it has to be of even length.
    
    From the Decomposition Theorem we know $\y_{\min}(G)=\y_{\min}(H)$, and by bipartiteness we know $\y_{\min}(H)=\y_{\min}(K)$. Now the only options for $\y_{\min}(G)$ are $-2$ and $-2\cos \Big (\frac \pi {m+1} \Big )$ for some $m$ (see \cite[Section 1.4.3, Section 1.4.4]{spectra}), respectively. This concludes the proof.
\end{proof}
The power of this result is that \emph{all} of the connected components for \emph{all} possible choices of pairs of colors for \emph{all} optimal colorings satisfying the requirement now are cycles, or paths on a certain fixed number of vertices, as the least eigenvalue of $G$ does not depend on the coloring or the component. The least eigenvalue of $G$ determines whether $K$ is a cycle or a path, and if it is a path, the length of the path.

In case $G$ is a tree or has an equal number of vertices and edges and the unique cycle is of odd length, then if $L(G)$ is Hoffman colorable the case $\y_{\min}(L(G))=-2$ is impossible because a cycle of alternating colors in $L(G)$ gives an even cycle in $G$, which does not exist.

With this in mind, we call a bipartite part $H$ of $G$ given some optimal edge coloring of $G$ such that $H$ is a path graph a \emph{maximal alternating path} of $G$, which we will abbreviate as \emph{MAP}. For an example of an MAP, see Figure \ref{fig:MAP}.

\begin{figure}[ht]
    \begin{center}
        \begin{tikzpicture}[scale=0.4]
            \coordinate (1) at (0,0);
            \coordinate (2) at (2,0);
            \coordinate (3) at (-1,1.73);
            \coordinate (4) at (-1,-1.73);
            \coordinate (5) at (-3,1.73);
            \coordinate (6) at (0,3.46);
            \coordinate (7) at (3,-1.73);
            \coordinate (8) at (-4,0);
            \coordinate (9) at (2,3.46);
            \coordinate (10) at (-4,3.46);
            \coordinate (1a) at (10,0);
            \coordinate (2a) at (12,0);
            \coordinate (3a) at (9,1.73);
            \coordinate (4a) at (9,-1.73);
            \coordinate (5a) at (7,1.73);
            \coordinate (6a) at (10,3.46);
            \coordinate (7a) at (13,-1.73);
            \coordinate (8a) at (6,0);
            \coordinate (9a) at (12,3.46);
            \coordinate (1b) at (14,0);
            \coordinate (2b) at (16,0);
            \coordinate (3b) at (13,1.73);
            \coordinate (4b) at (13,-1.73);
            \coordinate (5b) at (11,1.73);
            \coordinate (6b) at (14,3.46);
            \coordinate (7b) at (17,-1.73);
            \coordinate (8b) at (10,0);
            \coordinate (9b) at (16,3.46);

            \draw[red, thick] (1) -- (3);
            \draw[red, thick] (5) -- (8);
            \draw[blue, thick] (1) -- (4);
            \draw[blue, thick] (5) -- (3);
            \draw[blue, thick] (6) -- (9);
            \draw[blue, thick] (2) -- (7);
            \draw[green, thick] (6) -- (3);
            \draw[green, thick] (1) -- (2);
            \draw[green, thick] (5) -- (10);
            \draw[black, thick] (8a) -- (5a) -- (3a) -- (1a) -- (4a);
            \draw[black, thick] (6b) -- (3b) -- (1b) -- (2b);

            \filldraw[black] (1) circle (4pt) node[anchor=east]{1};
            \filldraw[black] (2) circle (4pt) node[anchor=west]{2};
            \filldraw[black] (3) circle (4pt) node[anchor=west]{3};
            \filldraw[black] (4) circle (4pt) node[anchor=east]{4};
            \filldraw[black] (5) circle (4pt) node[anchor=east]{5};
            \filldraw[black] (6) circle (4pt) node[anchor=east]{6};
            \filldraw[black] (7) circle (4pt) node[anchor=west]{7};
            \filldraw[black] (8) circle (4pt) node[anchor=east]{8};
            \filldraw[black] (9) circle (4pt) node[anchor=west]{9};
            \filldraw[black] (10) circle (4pt) node[anchor=east]{10};
            \filldraw[black] (1a) circle (4pt) node[anchor=west]{1};
            \filldraw[black] (3a) circle (4pt) node[anchor=west]{3};
            \filldraw[black] (4a) circle (4pt) node[anchor=west]{4};
            \filldraw[black] (5a) circle (4pt) node[anchor=east]{5};
            \filldraw[black] (8a) circle (4pt) node[anchor=east]{8};
            \filldraw[black] (1b) circle (4pt) node[anchor=east]{1};
            \filldraw[black] (2b) circle (4pt) node[anchor=west]{2};
            \filldraw[black] (3b) circle (4pt) node[anchor=east]{3};
            \filldraw[black] (6b) circle (4pt) node[anchor=east]{6};
        \end{tikzpicture}
    \end{center}
    \caption{Two maximal alternating paths in a tree.}
    \label{fig:MAP}
\end{figure}

What \thref{cor:LineGraphComponent} implies, is that every MAP of a tree/graph with a unique odd cycle $G$ with a Hoffman colorable line graph is of the same length, which is specified by the least eigenvalue of the line graph of $G$. The line graph of the tree from Figure \ref{fig:MAP} is therefore not Hoffman colorable.

\begin{thm}\thlabel{thm:linegraphs}
    Let $G$ be a connected graph with at least two edges. Then $L(G)$ is Hoffman colorable if and only if one of the following statements holds.
    \begin{enumerate}[label=(\roman*)]
        \item $G$ is 1-factorable;
        \item $G\cong K_{1,m}$ for some $m$;
        \item $G$ is a path graph;
        \item $G\cong K_3$;
        \item $G$ is the graph from Figure \ref{fig:K3withleaves}.
        
        \begin{figure}[ht]
        \begin{center}
        \begin{tikzpicture}[scale=0.4]
\coordinate (1) at (0,2);
\coordinate (2) at (1.73,-1);
\coordinate (3) at (-1.73,-1);
\coordinate (4) at (-0.87,-0.5);
\coordinate (5) at (0.87,-0.5);
\coordinate (6) at (0,1);
\draw[green,thick] (1) -- (6);
\draw[blue,thick] (2) -- (5);
\draw[red,thick] (3) -- (4);
\draw[green,thick] (4) -- (5);
\draw[blue,thick] (4) -- (6);
\draw[red,thick] (5) -- (6);
\filldraw[gray] (1) circle (4pt);
\filldraw[gray] (2) circle (4pt);
\filldraw[gray] (3) circle (4pt);
\filldraw[gray] (4) circle (4pt);
\filldraw[gray] (5) circle (4pt);
\filldraw[gray] (6) circle (4pt);
\end{tikzpicture}             
        \end{center}
        \caption{A sporadic Hoffman edge coloring.}
        \label{fig:K3withleaves}
        \end{figure}
    \end{enumerate}
\end{thm}

\begin{proof}
    Suppose that $G$ is a connected graph with at least two edges. The case where the least eigenvalue of the line graph is $-2$ is covered before, which leads to the 1-factorable graphs. So suppose that $G$ is a tree, or has a unique cycle, of odd length.

    First of all, note that the graphs from the list are all Hoffman colorable: the line graph of $K_{1,m}$ or $K_3$ is a complete graph so Hoffman colorable, the line graph of a path is still a path (Hoffman colorable by bipartiteness), and the line graph of the graph from Figure \ref{fig:K3withleaves} has largest eigenvalue $2\phi$ and least eigenvalue $-\phi$, where $\phi$ is the golden ratio.

    For the converse, assume that the line graph of $G$ is Hoffman colorable. Since $G$ has at least two edges and is connected, $\D(G)\ge 2$. If $\D(G)=2$, then $G$ must be a path or an odd cycle. Odd cycles of length at least 5 are not Hoffman colorable.

    For the remainder, assume that $\D(G)\ge 3$. Then note that for trees, starting with any root vertex, and coloring the edges in any available way in ascending order of distance of the root vertex, we get a valid $\D$-coloring of the edges. Furthermore, for the graphs with a unique odd cycle, we can color the cycle in any way we like using $\D\ge 3$ colors, and then color the remaining edges in any available way in ascending order of distance to the cycle, to get a valid $\D$-coloring. This shows that $\chi(L(G))=\D(G)$.

    Let $k$ be the length of the MAPs in $G$. Notice that we know that $k\ge 2$. Since $\D\ge 3$, there must be a leaf $\ell$ in $G$. Write $v$ for the neighbor of $\ell$. Since $k\ge 2$, $\{\ell,v\}$ cannot form an MAP. So, for every optimal edge coloring of $G$, each of the $\D(G)-1$ colors that are not the color of $\{\ell,v\}$ must be assigned to one edge from $v$, so that $\deg(v)=\D$.
    
    Suppose that $v$ is adjacent to another leaf $\ell'$. Then $(\ell,v,\ell')$ is an MAP, so $k=2$, which implies that $G$ is isomorphic to $K_{1,\D}$.

    Otherwise, if $v$ is adjacent to only one leaf, then $k\ge 3$. Now any vertex can only be adjacent to at most one leaf, and if a vertex is, then it must be of maximum degree $\D\ge 3$. The induced subgraph of $G$ on the set of non-leaves must now have minimum degree 2, and so it must be a cycle (since only one cycle can exist in the graph). Now we also know $\D=3$. So $G$ is an odd cycle graph, with one leaf added to each member of some non-empty subset of the vertices of the cycle.

    Suppose the cycle is of length at least 5, then color $G$ by using one color (say red) on one of the edges of the cycle, and two other colors (say green and blue) alternatingly on the remainder of the cycle. Then, there is a green-blue alternating path of length at least 4 (so $k\ge 4$), while there also exists a red-blue maximal alternating path of length at most 3 (the one from any blue-colored edge that does not meet the red edge on the cycle), which is a contradiction.

    So, the cycle must be of length 3. Attaching a leaf to just one or two of the vertices leads to an MAP of length 3 and a different MAP of length 4, which is a contradiction. The only possibility left is that $k=4$ and that $G$ is the graph from Figure \ref{fig:K3withleaves}, which concludes the proof.
\end{proof}

%%%%%%%%%%%%%%%%%%%%%%%%%%%%%%%%%%%%%%%%%%%%%%%%%%%%%%%%%%%%%%%%%%%%%%%%%%%%%%%%%%%%%%%%%%%%%%%%%%%%%%%%%%%%%
\section{The Composition Theorem}\label{sec:comp}
%%%%%%%%%%%%%%%%%%%%%%%%%%%%%%%%%%%%%%%%%%%%%%%%%%%%%%%%%%%%%%%%%%%%%%%%%%%%%%%%%%%%%%%%%%%%%%%%%%%%%%%%%%%%%
In this section, we introduce the Composition Theorem, and present an application of it to regular graphs of high enough valency.

%%%%%%%%%%%%%%%%%%%%%%%%%%%%%%%%%%%%%%%%%%%%%%%%%%%%%%%%%%%%%%%%%%%%%%%%%%%%%%%%%%%%%%%%%%%%%%%%%%%%%%%%%%%%%
\subsection{The Composition Theorem}
%%%%%%%%%%%%%%%%%%%%%%%%%%%%%%%%%%%%%%%%%%%%%%%%%%%%%%%%%%%%%%%%%%%%%%%%%%%%%%%%%%%%%%%%%%%%%%%%%%%%%%%%%%%%%

Recall that the Decomposition Theorem allows us to decompose a Hoffman colorable graph into induced subgraphs on choices of color classes. For the Composition Theorem, we argue in the other direction. We start with a Hoffman colorable template graph $T$, and extend it with an independent set $V_0$. In the case that we have edges between $T$ and $V_0$ in such a way that the structural requirements imposed by the Decomposition Theorem are satisfied, we would like to know when we get a Hoffman colorable graph. To this end, we introduce the Composition Theorem. In the following, if $x:A \to C$ and $y:B\to C$ are functions, then we write $x \sqcup y$ for the function from $A\sqcup B$ to $C$ applying either $x$ or $y$ accordingly.
\begin{thm}[Composition Theorem]\thlabel{thm:Comp}
Let $G$ be a graph with a $c+1$-coloring $V(G)=\bigsqcup _{i=0}^c V_i$. Write $T=G \setminus V_0$, and $H_i=G[V_0 \cup V_i]$ for $1\le i\le c$. Suppose the following hold.
    \begin{itemize}
        \item $T$ is Hoffman colorable with $c$ colors, with positive eigenvector $x$.
        \item For every $1\le i \le c$ we have $\y_{\min}(T)=\y_{\min}(H_i)$.
        \item There exists a vector $y:V_0 \to \R_{>0}$ such that for every $1\le i \le c$ we have that $ x |_{V_i} \sqcup y$ is a positive eigenvector of $H_i$.
    \end{itemize}
   Then $x \sqcup y$ is a positive eigenvector of $G$ for the eigenvalue $c\nu$. Furthermore, the following are equivalent.
    \begin{enumerate}[label=(\roman*)]
        \item $G$ is Hoffman colorable with $c+1$ colors;
        \item $\y_{\min}(G)=\y_{\min}(T)$;
        \item There exists no eigenvector $z$ of $G$ for an eigenvalue less than $\y_{\min}(T)$ such that $z|_{V_0}$ is orthogonal to $y$ and $z|_{V_i}$ is orthogonal to $x|_{V_i}$ for all $1\le i \le c$.
    \end{enumerate}
\end{thm}
\begin{proof}
   We first show that $a\coloneqq x \sqcup y$ is an eigenvector of $G$. We write $N_i(v)=N(v)\cap V_i$, and we write $\nu=-\y_{\min}(T)$. Note that the eigenvalue for the eigenvector $x$ of $T$ is $(c-1)\nu$, and the eigenvalue for the eigenvector $x|_{V_i} \sqcup y$ of $H_i$ is $\nu$. If $u\in V_0$, then
    \begin{align*}
        \sum_{v\in N(u)} a(v) &=\sum_{i=1}^c \sum_{v\in N_i(u)} a(v),
        \intertext{and since $x |_{V_i} \sqcup y$ is an eigenvector of $H_i$, we have}
        &= \sum_{i=1}^c \nu a(u)=c\nu a(u).
        \intertext{Now if $u\in V_i$ with $1\le i \le c$, we have}
        \sum_{v\in N(u)} a(v) &= \sum_{v\in N_0(u)} a(v) + \sum_{v\in N_T(u)} x(v),
        \intertext{where we use that $x|_{V_i} \sqcup y$ and $x$ are eigenvectors to get}
        &= \nu a(u) + (c-1)\nu a(u)=c \nu a(u).
    \end{align*}
    It follows that $a$ is an eigenvector of $G$ for the eigenvalue $c\nu$. Note that $a$ is positive, so that $c\nu$ is the largest eigenvalue of $G$.

   With similar reasoning (also applying the Decomposition Theorem on bipartite parts of $T$), for $1\le i \le c$ the vector $b_i \coloneqq x|_{V_i} \sqcup -y \sqcup \underline{0}$, where $\underline{0}$ is the zero vector on $V(T) \setminus V_i$, is an eigenvector of $G$ for the eigenvalue $-\nu$.

    So we have $\y_{\min}(G) \le -\nu$. Hence $h(G)\le c+1$. Now the $(c+1)$-coloring of the statement is a Hoffman coloring if and only if $h(G) = c+1$, if and only if $\y_{\min}(G)=-\nu$. Since the vectors $a$ and $b_i$ have an eigenvalue at least $-\nu$, it is therefore sufficient to only consider eigenvectors orthogonal to the space generated by $a$ and $b_i$. This space is also generated by the vectors $y \sqcup \underline{0}$ and $x|_{V_i} \sqcup \underline{0}$. This concludes the proof.
\end{proof}

%%%%%%%%%%%%%%%%%%%%%%%%%%%%%%%%%%%%%%%%%%%%%%%%%%%%%%%%%%%%%%%%%%%%%%%%%%%%%%%%%%%%%%%%%%%%%%%%%%%%%%%%%%%%%
\subsection{Regular Hoffman colorable graphs}\label{sec:cons:regular}
%%%%%%%%%%%%%%%%%%%%%%%%%%%%%%%%%%%%%%%%%%%%%%%%%%%%%%%%%%%%%%%%%%%%%%%%%%%%%%%%%%%%%%%%%%%%%%%%%%%%%%%%%%%%%
In this section, we give a sufficient condition for a regular graph to be Hoffman colorable. We know that such a graph has to satisfy \thref{prop:constantequitable}. In order to speak more easily about it, we say that a coloring of a regular graph $G$ is \emph{$\nu$-equitable} if every vertex is adjacent to precisely $\nu$ vertices of every color other than its own. If a coloring is $\nu$-equitable, then necessarily all its color classes are of the same size. Note that \thref{prop:constantequitable} now says that every Hoffman coloring of a regular graph $G$ is $(-\y_{\min}(G))$-equitable.

In \thref{prop:colorcomplementregular} we show that if a regular graph $G$ has a $\nu$-equitable coloring and $\nu$ is big enough, then $G$ is Hoffman colorable. In order to prove that, we need the following notion. Given a coloring of a graph $G$, we define the \emph{color complement} $\ol G_{\text{color}}$ to be the graph on the same vertex set, such that vertices are adjacent in $\ol G_{\text{color}}$ whenever they were not adjacent in $G$, and belong to different color classes. The coloring of $G$ used for the color complement is also a valid coloring for the color complement, by construction. If we use a $\nu$-equitable coloring of $G$, the coloring is $(m-\nu)$-equitable for $\ol G_{\text{color}}$, where $m$ is the size of the color classes.

As an example (see Figure \ref{fig:colorcomplement}), take $G$ to be the 6-cycle with a 1-equitable 3-coloring. Then the color complement is the disjoint union of two complete graphs of size three, optimally colored by a 2-equitable 3-coloring. Note that a different coloring of the 6-cycle would give a different color complement.
    
    \begin{figure}[ht]
    \begin{center}
        \begin{tikzpicture}[scale=0.4]
            \coordinate (1) at (2,0);
            \coordinate (2) at (1,1.7);
            \coordinate (3) at (-1,1.7);
            \coordinate (4) at (-2,0);
            \coordinate (5) at (-1,-1.7);
            \coordinate (6) at (1,-1.7);
            \draw[gray, thick] (1) -- (2) -- (3) -- (4) -- (5) -- (6) -- (1);
            \filldraw[blue] (1) circle (4pt);
            \filldraw[red] (2) circle (4pt);
            \filldraw[green] (3) circle (4pt);
            \filldraw[blue] (4) circle (4pt);
            \filldraw[red] (5) circle (4pt);
            \filldraw[green] (6) circle (4pt);
            \coordinate (1a) at (8,0);
            \coordinate (2a) at (7,1.7);
            \coordinate (3a) at (5,1.7);
            \coordinate (4a) at (4,0);
            \coordinate (5a) at (5,-1.7);
            \coordinate (6a) at (7,-1.7);
            \draw[gray, thick] (1a) -- (3a) -- (5a) -- (1a);
            \draw[gray, thick] (2a) -- (4a) -- (6a) -- (2a);
            \filldraw[blue] (1a) circle (4pt);
            \filldraw[red] (2a) circle (4pt);
            \filldraw[green] (3a) circle (4pt);
            \filldraw[blue] (4a) circle (4pt);
            \filldraw[red] (5a) circle (4pt);
            \filldraw[green] (6a) circle (4pt);
        \end{tikzpicture}
    \end{center}
    \caption{A pair of color complemented graphs.}\label{fig:colorcomplement}
    \end{figure}
\begin{prop}\thlabel{prop:colorcomplementregular}
    Let $G$ be a graph with a $\nu$-equitable $c$-coloring with every color class of size $m$. If $\nu \ge m (c-1)/c$, then $G$ is Hoffman colorable and $\chi(G)=c$.
\end{prop}
\begin{proof}
    We argue by induction on $c$. The base case $c=2$ is trivial as every bipartite graph is Hoffman colorable. Suppose the statement holds for $c$, and suppose $G$ has a $\nu$-equitable $(c+1)$-coloring $V(G)=\bigsqcup _{i=0}^c V_i$, and that $\nu \ge m c/(c+1)$. We apply the Composition Theorem, where $x$ and $y$ are constant vectors. By the induction hypothesis, $T$ is Hoffman colorable, and $\y_{\min}(T)=-\nu=\y_{\min}(H_i)$, so that the requirements of the Composition Theorem are met.
    
    Let $z$ be an eigenvector of $G$ with eigenvalue $\y$ as in the Composition Theorem, then $z$ sums to zero on every color class. Note that this implies that $z$ is also an eigenvector of the color complement of $G$, with eigenvalue $-\y$. Note that the color complement of $G$ is of valency $c(m-\nu)$, so that $-\y \le c(m-\nu)$. By $\nu \ge mc/(c+1)$ we get $c(m-\nu)\le \nu$. Hence, $-\y\le \nu$ and the result follows by the Composition Theorem.
\end{proof}
For example, for 3-colorable graphs, if $\nu$ is at least two thirds of the color class size, then the graph is Hoffman colorable. Furthermore, note that if $c=m=\nu+1$, then the requirement of \thref{prop:colorcomplementregular} is automatically satisfied. A graph with such a coloring can be obtained by removing a perfect matching from every bipartite part of a regular complete multipartite graph. This way, we get exponentially many non-isomorphic regular Hoffman colorable graphs.

%%%%%%%%%%%%%%%%%%%%%%%%%%%%%%%%%%%%%%%%%%%%%%%%%%%%%%%%%%%%%%%%%%%%%%%%%%%%%%%%%%%%%%%%%%%%%%%%%%%%%%%%%%%%%
\section{The algorithm}\label{sec:algorithm}
%%%%%%%%%%%%%%%%%%%%%%%%%%%%%%%%%%%%%%%%%%%%%%%%%%%%%%%%%%%%%%%%%%%%%%%%%%%%%%%%%%%%%%%%%%%%%%%%%%%%%%%%%%%%%

In this section we propose an algorithm for computing every connected Hoffman colorable graph given a number of vertices and a number of colors. The algorithm is generative in nature and is based on the Decomposition and Composition Theorems. Let $H_{i,j}$ be the bipartite part on color classes $i$ and $j$. Then if $i,j,k$ are three distinct color classes, by the Decomposition Theorem we have $\y_{\max}(H_{i,j})=\y_{\max}(H_{i,k})$ and furthermore, $H_{i,j}$ and $H_{i,k}$ each have a positive eigenvector, and they agree on the $i$'th color class. In this way, the bipartite graphs are \emph{compatible}. We call the collection of all $\binom{c}{2}$ bipartite parts of a $c$-coloring a \emph{collection of compatible bipartite parts}.

The general idea of the algorithm is to generate bipartite graphs (\cite{geng}) and apply the Composition Theorem inductively. To check compatibility, we compute the largest eigenvalue and a corresponding positive eigenvector for the bipartite graphs. See \ref{algo} for a pseudocode outline of the algorithm. We use the Composition Theorem to check if the obtained graph is indeed Hoffman colorable. The Decomposition Theorem ensures that every connected Hoffman colorable graph is found by the algorithm.

\begin{algorithm}[ht]
    \caption{Algorithm for computing all connected Hoffman colorable graphs}
    \label{algo}
    \KwIn{A number of vertices $n$ and a number of colors $\chi$;}
    \KwOut{A sequence \emph{gr} of connected Hoffman colorable graphs, and a sequence \emph{disc} of disconnected cases;}
    \nlset{1a}Form all viable integer partitions of $n$ into $\chi$ parts\;
    \nlset{1b}Generate all bipartite parts of relevant sizes and sort by largest eigenvalue\;
    \For{eigenvalue}{
    \nlset{2a}Eliminate integer partitions for which there exists a pair with no possible bipartite parts\;
    \nlset{2b}Filter out the disconnected bipartite parts\;
    \For{connected bipartite part}{
    Compute Perron eigenvector\;
    Sort graph by Perron eigenvector\;
    }
    \For{integer partition}{
    \nlset{3a}\If{one possible bipartite part is disconnected}{
    Append this integer partition and all relevant info to \emph{disc}\;}\
    \nlset{3b}Form every possible collection of compatible bipartite parts out of connected bipartite parts\;
    \For{collection of compatible bipartite parts}{
    \nlset{4}Compose in every possible way\;
    \If{composed graph is Hoffman colorable}{Append to \emph{gr}\;}}}}
    \KwRet{ \emph{gr}, \emph{disc}.}
\end{algorithm}

Note that disconnected bipartite parts pose a problem, as the dimension of the eigenspace for the largest eigenvalues might be greater than one; in other words, there is not a unique positive eigenvector. The algorithm automatically isolates these cases for human intervention to take place. In the smallest cases we were able to solve these by hand. 

The algorithm has been implemented in the computer algebra system Magma (\cite{magma}), version V.28-8. The code can be found via GitHub \cite{github}.

%%%%%%%%%%%%%%%%%%%%%%%%%%%%%%%%%%%%%%%%%%%%%%%%%%%%%%%%%%%%%%%%%%%%%%%%%%%%%%%%%%%%%%%%%%%%%%%%%%%%%%%%%%%%%
\subsection{Results}

We provide four tables to present the results of applying the algorithm to three (Table \ref{tab:three}), four (Table \ref{tab:four}), five (Table \ref{tab:five}) and six colors (Table \ref{tab:six}). For every choice of input of a number of colors and a number of vertices, we provide the total number of connected Hoffman colorable graphs (up to isomorphism), the number of regular ones and irregular ones, and the number of graphs where the Hoffman bound outperforms (\ref{classicbound}) rounded up.

Some cells in the tables have a $\ge$-sign, caused by disconnected bipartite parts as described above. In other cases we can be sure that we have found every connected Hoffman colorable with the given number of vertices and colors.

\begin{table}
\begin{center}
\begin{tabular}{|c||c|c|c|c|}
    \hline
        \#vertices& \#graphs & \#regulars  &\#irregulars &\#outperforming \\
         &  & &  & (\ref{classicbound}) \\
        \hline
        \hline
        3 & 1 & 1 & 0 & 0 \\
        \hline
        4 & 0 & - & - & -  \\
        \hline
        5 & 0 & - & - & -  \\
        \hline
        6 & 2 & 1 & 1 & 1 \\
        \hline
        7 & 0 & - & - & - \\
        \hline
        8 & 0 & - & - & - \\
        \hline
        9 & 13 & 4 & 9 & 3  \\
        \hline
        10 & 3 & - & 3 & 0  \\
        \hline
        11 & 2 & - & 2 & 0  \\
        \hline
        12 & 68 & 16 & 52 & 20 \\
        \hline
        13 & 14 & - & 14 & 3 \\
        \hline
        14 & 46 & - & 46 & 10 \\
        \hline
        15 & $\ge 1634$ & $\ge 900$ & $\ge 734$ & $\ge 31$ \\
        \hline
    \end{tabular}
\end{center}
\caption{Connected Hoffman colorable graphs with three colors.}
\label{tab:three}
\end{table}

\begin{table}
\begin{center}
    \begin{tabular}{|c||c|c|c|c|}
    \hline
        \#vertices& \#graphs & \#regulars  &\#irregulars &\#outperforming \\
         &  & &  & (\ref{classicbound}) \\
        \hline
        \hline
        4 & 1 & 1 & 0 & 0 \\
        \hline
        5 & 0 & - & - & - \\
        \hline
        6 & 0 & - & - & - \\
        \hline
        7 & 0 & - & - & - \\
        \hline
        8 & 1 & 1 & 0 & 0 \\
        \hline
        9 & 0 & - & - & - \\
        \hline
        10 & 0 & - & - & -\\
        \hline
        11 & 2 & - & 2 & 2\\
        \hline
        12 & 8 & 5 & 3 & 1 \\
        \hline
        13 & 17 & - & 17 & 0 \\
        \hline
        14 & 5 & - & 5 & 5 \\
        \hline
        15 & 10 & - & 10 & 10 \\
        \hline
        16 & $\ge 167$ & $\ge 92$ & $\ge 75$ & $\ge 12$ \\
        \hline
        17 & 8 & - & 8 & 8 \\
        \hline
        18 & $\ge 380$ & - & $\ge 380$ & $\ge 360$ \\
        \hline
    \end{tabular}
\end{center}
\caption{Connected Hoffman colorable graphs with four colors.}
\label{tab:four}
\end{table}

\begin{table}
\begin{center}
    \begin{tabular}{|c||c|c|c|c|}
    \hline
        \#vertices& \#graphs & \#regulars  &\#irregulars &\#outperforming \\
         &  & &  & (\ref{classicbound}) \\
        \hline
        \hline
        5 & 1 & 1 & 0 & 0\\
        \hline
        6 & 0 & - & - & - \\
        \hline
        7 & 0 & - & - & - \\
        \hline
        8 & 0 & - & - & - \\
        \hline
        9 & 0 & - & - & - \\
        \hline
        10 & 1 & 1 & 0 & 0 \\
        \hline
        11 & 0 & - & - & - \\
        \hline
        12 & 0 & - & - & - \\
        \hline
        13 & 2 & - & 2 & 2 \\
        \hline
        14 & 0 & - & - & - \\
        \hline
        15 & 10 & 7 & 3 & 0 \\
        \hline
        16 & 16 & - & 16 & 16 \\
        \hline
        17 & 34 & - & 34 & 0 \\
        \hline
        18 & $\ge 5$ & - & $\ge 5$ & $\ge 5$ \\
        \hline
        19 & 7 & - & 7 & 7 \\
        \hline
    \end{tabular}
\end{center}
\caption{Connected Hoffman colorable graphs with five colors.}
\label{tab:five}
\end{table}

\begin{table}
\begin{center}
    \begin{tabular}{|c||c|c|c|c|}
    \hline
        \#vertices& \#graphs & \#regulars  &\#irregulars &\#outperforming \\
         &  & &  & (\ref{classicbound}) \\
        \hline
        \hline
        6 & 1 & 1 & 0 & 0 \\
        \hline
        7 & 0 & - & - & - \\
        \hline
        8 & 0 & - & - & - \\
        \hline
        9 & 0 & - & - & - \\
        \hline
        10 & 0 & - & - & - \\
        \hline
        11 & 0 & - & - & - \\
        \hline
        12 & 1 & 1 & 0 & 0 \\
        \hline
        13 & 0 & - & - & - \\
        \hline
        14 & 0 & - & - & - \\
        \hline
        15 & 1 & - & 1 & 1 \\
        \hline
        16 & 0 & - & - & - \\
        \hline
        17 & 0 & - & - & - \\
        \hline
        18 & 10 & 5 & 5 & 3 \\
        \hline
        19 & 8 & - & 8 & 8 \\
        \hline
        20 & 0 & - & - & - \\
        \hline
        21 & & & &  \\
        \hline
        22 & 5 & - & 5 & 5 \\
        \hline
        23 & 17 & - & 17 & 17 \\
        \hline
    \end{tabular}
\end{center}
\caption{Connected Hoffman colorable graphs with six colors.}
\label{tab:six}
\end{table}

\newpage

%%%%%%%%%%%%%%%%%%%%%%%%%%%%%%%%%%%%%%%%%%%%%%%%%%%%%%%%%%%%%%%%%%%%%%%%%%%%%%%%%%%%%%%%%%%%%%%%%%%%%%%%%%%%%
\subsection*{Acknowledgements} 
%%%%%%%%%%%%%%%%%%%%%%%%%%%%%%%%%%%%%%%%%%%%%%%%%%%%%%
Aida Abiad is supported by NWO (Dutch Research Council) through the grants VI.Vidi.213.085 and OCENW.KLEIN.475. We are grateful to the anonymous referee for the careful reading and comments, which improved the presentation of the paper.

%%%%%%%%%%%%%%%%%%%%%%%%%%%%%%%%%%%%%%%%%%%%%%%%%%%%%%

\end{document}